\documentclass{article}%
\usepackage{mathrsfs}
\usepackage{amsmath}
\usepackage{amsfonts}
\usepackage{amssymb}
\usepackage{graphicx}%
\providecommand{\U}[1]{\protect \rule{.1in}{.1in}}
\newtheorem{theorem}{Theorem}[section]

\newtheorem{corollary}[theorem]{Corollary}

\newtheorem{definition}[theorem]{Definition}
\newtheorem{example}[theorem]{Example}

\newtheorem{lemma}[theorem]{Lemma}

\newtheorem{proposition}[theorem]{Proposition}
\newtheorem{remark}[theorem]{Remark}

\newenvironment{proof}[1][Proof]{\noindent \textbf{#1.} }{\  \rule{0.5em}{0.5em}}
\RequirePackage{CJK}
\AtBeginDocument{\begin{CJK*}{GBK}{song}\CJKtilde}
\AtEndDocument{\end{CJK*}}
\begin{document}

\title{G-Gaussian Processes under Sublinear Expectations and $q$-Brownian Motion in
Quantum Mechanics}
\author{Shige Peng\\School of Mathematics, Shandong University, China}
\date{May 05, 2011}
\maketitle

\abstract{We provide a general approach to construct a stochastic
process with a given consistent family of  finite dimensional
distributions under a nonlinear expectation space. We use this
approach to construct a generalized Gaussian process under a
sublinear expectation and a $q$-Brownian motion, under a
complex-valued linear expectation, with which a new type of
Feynman-Kac formula can be derived to represent the solution of a
Schr\"odinger equation.}

\section{Introduction}

In this paper we are concerned with constructions of some typical stochastic
processes under a generalized notion of expectation. We have already
introduced, in [Peng2004-2010], $G$-Markovian processes, $G$-Brownian motions
and the related stochastic calculus under nonlinear expectation space (see
also [Hu-Peng2010] for $G$-L\'{e}vy processes). The main idea is to directly
use a nonlinear (sublinear) expectation to valuate the random variables
depending on the paths of the corresponding stochastic process. For example, a
generalized Brownian motion $(B_{t})_{t\geq0}$, called $G$-Brownian motion, is
merely a process with continuous paths under a specifically designed nonlinear
expectation so that each increments of $(B_{t})_{t\geq0}$ is stable and
independent of its historical path. It was proved that each of its increment
is $G$-normally distributed random variables. On the other hand, such type of
$G$-normal distributions can be also obtained as the limit in law of a sum of
an i.i.d. sequence of random variables.

In this paper we will construct two types of stochastic processes. The first
one is a Gaussian process under a sublinear expectation. All finite
dimensional distributions of this process are $G$-normally distributed. When
we deduce to a classical framework the corresponding expectation is linear and
thus the $G$-normal distributions become the classical normal distributions.
In this case this process is noting but a classical Gaussian. Quite different
from a classical situation, in general, a $G$-Brownian motion $(B_{t}%
)_{t\geq0}$ is not a Gaussian process although each increment of a
$G$-Brownian motion is normally distributed.

The second type of stochastic process is constructed under a complex-valued
linear expectation in the place of a real valued expectation. Under such
$\mathbb{C}$-valued expectation we can define a new type of random variable
which is $q$-normally distribution and then a $q$-Brownian motion, where $q$
stands for ``quantum''. Under this framework the solution of the well-known
Schr\"odinger equation can be expressed by a new formula of Feynman-Kac type.
We then get a clear path picture via this $q$-Brownian motion.

The paper is organized as follows. In section 2 we present some preliminary
notations and results of a general framework of nonlinear expectations and
nonlinear expectation spaces, including nonlinear distribution, independence,
$G$-normal distribution. In Section 3 we give a general construction of a
stochastic process with a given consistent family of finite dimensional
distributions. It is in fact a generalization of the well-know arroach of
Kolmogorov's consistency under a nonlinear expectation. We also discuss how to
find the upper expectation of a family of probability measures for some given
sublinear expectation and under what condition we can obtain the continuity of
stochastic process. In Section 4 the construction of a $G$-Gaussian process
are introduced. In the last section we introduce a complex valued linear
expectation under which the $q$-normal random variable and the $q$-Brownian
motion are defined. We also provide an Appendix for the convenience to read
the paper.

\section{Preliminaries: sublinear expectation}

A sublinear expectation is also called an upper expectation. It is frequently
applied to situations when the probability models themselves have uncertainty.
In this section, we present the basic notion of sublinear expectations and the
corresponding sublinear expectation spaces.

\subsection{{Nonlinear expectation}}

Let $\Omega$ be a given set and let $\mathcal{H}$ \label{huah}be a linear
space of real valued functions defined on $\Omega$ containing constants. In
this paper, we often suppose that $|X|\in \mathcal{H}$ if $X\in \mathcal{H}$.
The space $\mathcal{H}$ is also called the space of random variables.

\begin{definition}
{\label{Def-1 copy(1)} {A \textbf{Sublinear expectation }\label{sube}
\index{Sublinear expectation}%
$\mathbb{E}$ is a functional $\mathbb{{E} }:\mathcal{H}\rightarrow \mathbb{R}$
satisfying }}

\noindent{\textbf{\textup{(i)} Monotonicity: }%
\[
\mathbb{E}[X]\geq \mathbb{E}[Y] \  \  \text{if}\ X\geq Y.
\]
}

\noindent \textbf{\textup{(ii)} Constant preserving:}
\[
\mathbb{E}[c]=c\  \  \  \text{for}\ c\in \mathbb{R}.
\]

\noindent \textbf{\textup{(iii)} Sub-additivity: } For each $X,Y\in
{\mathcal{H}}$,
\[
\mathbb{E}[X+Y]\leq \mathbb{E}[X]+\mathbb{E}[Y].
\]

\noindent \textbf{\textup{(iv)} Positive homogeneity:}
\[
\mathbb{E}[\lambda X]=\lambda \mathbb{E}[X]\  \  \  \text{for}\  \lambda
\geq0\text{.}%
\]

The triple $(\Omega,\mathcal{H},\mathbb{E}\mathbb{)}$ is called a
\textbf{sublinear expectation space}. If only \textup{(i)} and \textup{(ii)}
are satisfied, {{$\mathbb{E}$ is called a }}\textbf{nonlinear expectation} and
the triple $(\Omega,\mathcal{H},\mathbb{E}\mathbb{)}$ is called a
\textbf{nonlinear expectation space}.
\end{definition}

\begin{remark}
In more general situation the above {{$\mathbb{{E}}$ may be {$\mathbb{R}^{k}%
$-valued, namely, the functions of $\mathcal{H}$ are $\mathbb{R}^{k}$-valued
and }$\mathbb{{E}}$}} maps {{$\mathcal{H}$}} to {{$\mathbb{R}^{k}$. For linear
situation we usually only assume $\mathbb{E}[c]=c$, for }}$c\in \mathbb{R}^{d}$
and $\mathbb{E}[\alpha X+\beta Y]=\alpha \mathbb{E}[X]+\beta \mathbb{E}[Y]$, for
$\alpha,\beta \in \mathbb{R}$ and $X$, $Y\in \mathcal{H}$. $\mathbb{E}$ is called
an $\mathbb{R}^{d}$-valued linear expectation, or simply $\mathbb{R}^{d}%
$-expectation. A type of $\mathbb{C}$-valued linear expectation will be
discussed in the Section 5.
\end{remark}

\begin{definition}
Let $\mathbb{E}_{1}$ and $\mathbb{E}_{2}$ be two nonlinear expectations
defined on $(\Omega,\mathcal{H})$. $\mathbb{E}_{1}$ is said to be
\textbf{dominated} by $\mathbb{E}_{2}$ if
\[
\mathbb{E}_{1}[X]-\mathbb{E}_{1}[Y]\leq \mathbb{E}_{2}[X-Y]\  \ ~\text{for}%
\ X,Y\in \mathcal{H}.
\]

\end{definition}

\begin{remark}
If the inequality in \textup{(iii)} becomes equality, then {{$\mathbb{E}$}} is
a linear expectation, i.e., {{$\mathbb{E}$}} is a linear functional satisfying
\textup{(i)} and \textup{(ii)}.
\end{remark}

\begin{remark}
\textup{(iii)+(iv)} is called \textbf{sublinearity}. This sublinearity implies

\noindent \textbf{\textup{(v)}} \textbf{Convexity}:
\[
\mathbb{E}[\alpha X+(1-\alpha)Y]\leq \alpha \mathbb{E}[X]+(1-\alpha
)\mathbb{E}[Y]\  \  \text{for}\  \alpha \in \lbrack0,1].
\]
If a nonlinear expectation $\mathbb{E}$ satisfies convexity, we call it a
\textbf{convex expectation}.

The properties \textup{(ii)}\textup{+}\textup{(iii)} implies

\noindent \textup{\textbf{(vi)}} \textbf{Constant translatability}:
\[
\mathbb{E}[X+c]=\mathbb{E}[X]+c\  \  \text{for}\ c\in \mathbb{R}.
\]
In fact, we have
\[
\mathbb{E}[X]+c=\mathbb{E}[X]-\mathbb{E}[-c]\leq \mathbb{E}[X+c]\leq
\mathbb{E}[X]+\mathbb{{E}}[c]=\mathbb{E}[X]+c.
\]

For property \textup{(iv)}, an equivalence form is
\[
{\mathbb{E}}[\lambda X]=\lambda^{+}{\mathbb{E}}[X]+\lambda^{-}{\mathbb{E}}[-X]
\ ~\text{for}\  \lambda \in \mathbb{R}.
\]

\end{remark}

In this paper, we mainly consider the following type of nonlinear expectation
spaces $(\Omega,\mathcal{H},\mathbb{E}\mathbb{)}$: if $X_{1},\cdots,X_{n}%
\in \mathcal{H}$ then $\varphi(X_{1},\cdots,X_{n})\in \mathcal{H}$ for each
$\varphi \in C_{l.Lip}(\mathbb{R}^{n})$ where $C_{l.Lip}(\mathbb{R}^{n})$
denotes the linear space of functions $\varphi$ satisfying
\begin{align*}
|\varphi(x)-\varphi(y)|  &  \leq C(1+|x|^{m}+|y|^{m})|x-y|\  \  \text{for}%
\ x,y\in \mathbb{R}^{n}\text{, \ }\\
\  &  \text{some }C>0\text{, }m\in \mathbb{N}\text{ depending on }\varphi.
\end{align*}
In this case $X=(X_{1},\cdots,X_{n})$ is called an $n$-dimensional random
vector, denoted by $X\in \mathcal{H}^{n}$.

\begin{remark}
It is clear that if $X\in \mathcal{H}$ then $|X|$, $X^{m}\in \mathcal{H}$. More
generally, $\varphi(X)\psi(Y)\in \mathcal{H}$ if $X,Y\in \mathcal{H}$ and
$\varphi,\psi \in C_{l.Lip}(\mathbb{R})$. In particular, if $X\in \mathcal{H}$
then $\mathbb{E}[|X|^{n}]<\infty$ for each $n\in \mathbb{N}$.
\end{remark}

Here we use $C_{l.Lip}(\mathbb{R}^{n})$ in our framework only for some
convenience of techniques. In fact our essential requirement is that
$\mathcal{H}$ contains all constants and, moreover, $X\in \mathcal{H}$ implies
$\left \vert X\right \vert \in \mathcal{H}$. In general, $C_{l.Lip}%
(\mathbb{R}^{n})$ can be replaced by any one of the following spaces of
functions defined on $\mathbb{R}^{n}$.

\begin{itemize}
\item { $\mathbb{L}^{\infty}(\mathbb{R}^{n})$: the space of bounded
Borel-measurable functions; }

\item { $C_{b}(\mathbb{R}^{n})$: the space of bounded and continuous
functions; \label{cbr}}

\item {$C_{b}^{k}(\mathbb{R}^{n})$: the space of bounded and $k$-time
continuously differentiable functions with bounded derivatives of all orders
less than or equal to $k$;\label{cbk}}

\item { $C_{unif}(\mathbb{R}^{n})$: the space of bounded and uniformly
continuous functions; \label{cu}}

\item { $C_{b.Lip}(\mathbb{R}^{n})$: the space of bounded and Lipschitz
continuous functions; \label{cbl}}

\item { $L^{0}(\mathbb{R}^{n})$: the space of Borel measurable functions.
\label{llll0}}
\end{itemize}

\subsection{Representation of a sublinear expectation}

A sublinear expectation can be expressed as a supremum of linear expectations.

\begin{theorem}
\label{t1} Let {$\mathbb{E}$} be a functional defined on a linear space
$\mathcal{H}$ satisfying sub-additivity and positive homogeneity. Then there
exists a family of linear functionals $\{E_{\theta}:\theta \in \Theta \}$ defined
on $\mathcal{H}$ such that
\[
{\mathbb{E}}[X]=\sup_{\theta \in \Theta}E_{\theta}[X]\  \  \text{for}%
\ X\in \mathcal{H}%
\]
and, for each $X\in \mathcal{H}$, there exists $\theta_{X}\in \Theta$ such that
{$\mathbb{E}$}$[X]=E_{\theta_{X}}[X]$.

Furthermore, if $\mathbb{E}$ is a sublinear expectation, then the
corresponding $E_{\theta}$ is a linear expectation.
\end{theorem}

\begin{remark}
\label{r1} It is important to observe that the above linear expectation
$E_{\theta}$ is only assumed to be finitely additive. But we can apply the
well-known Daniell-Stone Theorem to prove that there is a unique $\sigma
$-additive probability measure $P_{\theta}$ on $(\Omega,\sigma(\mathcal{H}))$
such that%
\[
E_{\theta}[X]=\int_{\Omega}X(\omega)dP_{\theta},\  \  \ X\in \mathcal{H}\text{.}%
\]
The corresponding model uncertainty of probabilities is the subset
$\{P_{\theta}:\theta \in \Theta \}$, and the corresponding uncertainty of
distributions for an $n$-dimensional random vector $X$ in $\mathcal{H}$ is
$\{F_{X}(\theta,A):=P_{\theta}(X\in A):A\in \mathcal{B}({\mathbb{R}}^{n})\}$.
\end{remark}

\subsection{Distributions, independence and product spaces}

We now give the notion of distributions of random variables under nonlinear expectations.

Let $X=(X_{1},\cdots,X_{n})$ be a given $n$-dimensional random vector on a
nonlinear expectation space $(\Omega,\mathcal{H},\mathbb{E})$. We define a
functional on $C_{l.Lip}(\mathbb{R}^{n})$ by
\[
\mathbb{F}^{X}[\varphi]:=\mathbb{E}[\varphi(X)]:\varphi \in C_{l.Lip}%
(\mathbb{R}^{n})\rightarrow \mathbb{R}.
\]
The triple $(\mathbb{R}^{n},C_{l.Lip}(\mathbb{R}^{n}),\mathbb{F}_{X})$ forms a
nonlinear expectation space. $\mathbb{F}_{X}$ is called the
\textbf{distribution} of $X$ under $\mathbb{E}$. In the $\sigma$-additive
situation (see Remark \ref{r1}), we have the following form:%
\[
\mathbb{F}_{X}[\varphi]=\sup_{\theta \in \Theta}\int_{\mathbb{R}^{n}}%
\varphi(x)F_{X}(\theta,dx).
\]

\begin{definition}
\label{d1} Let $X_{1}$ and $X_{2}$ be two $n$--dimensional random vectors
defined on {nonlinear expectation spaces }$(\Omega_{1},\mathcal{H}%
_{1},\mathbb{{E}}_{1})${ and }$(\Omega_{2},\mathcal{H}_{2},\mathbb{E}_{2})$,
respectively. They are called \textbf{identically distributed}, denoted by
$X_{1}\overset{d}{=}X_{2}$ or $\mathbb{F}^{X_{1}}=\mathbb{F}^{X_{2}}$, if
\[
\mathbb{E}_{1}[\varphi(X_{1})]=\mathbb{E}_{2}[\varphi(X_{2})]\  \  \  \text{for}%
\  \varphi \in C_{l.Lip}(\mathbb{R}^{n}).
\]
It is clear that $X_{1}\overset{d}{=}X_{2}$ if and only if their distributions
coincide. We say that the distribution of $X_{1}$ is stronger than that of
$X_{2}$ if $\mathbb{E}_{1}[\varphi(X_{1})]\geq \mathbb{E}_{2}[\varphi(X_{2})]$,
for each $\varphi \in C_{l.Lip}(\mathbb{R}^{n})$.
\end{definition}

\begin{remark}
In the case of sublinear expectations, $X_{1}\overset{d}{=}X_{2}$ implies that
the uncertainty subsets of distributions of $X_{1}$ and $X_{2}$ are the same,
e.g., in the framework of Remark \ref{r1},
\[
\{F_{X_{1}}(\theta_{1},\cdot):\theta_{1}\in \Theta_{1}\}=\{F_{X_{2}}(\theta
_{2},\cdot):\theta_{2}\in \Theta_{2}\}.
\]
Similarly if the distribution of $X_{1}$ is stronger than that of $X_{2}$,
then
\[
\{F_{X_{1}}(\theta_{1},\cdot):\theta_{1}\in \Theta_{1}\} \supset \{F_{X_{2}%
}(\theta_{2},\cdot):\theta_{2}\in \Theta_{2}\}.
\]

\end{remark}

The distribution of $X\in \mathcal{H}$ has the following four typical
parameters:
\[
\bar{\mu}:={\mathbb{E}}[X],\  \  \underline{\mu}:=-\mathbb{{E}}%
[-X],\  \  \  \  \  \  \  \  \bar{\sigma}^{2}:={\mathbb{E}}[X^{2}],\  \  \underline
{\sigma}^{2}:=-{\mathbb{E}}[-X^{2}].\  \
\]
The intervals $[\underline{\mu},\bar{\mu}]$ and $[\underline{\sigma}^{2}%
,\bar{\sigma}^{2}]$ characterize the \textbf{mean-uncertainty} and the
\textbf{variance-uncertainty} of $X$ respectively.

The following property is very useful in our sublinear expectation theory.

\begin{proposition}
{ \label{Prop-X+Y}Let }$(\Omega,\mathcal{H},\mathbb{E})$ be a sublinear
expectation space and { $X,Y{\ }$be two random variables such that
$\mathbb{E}[Y]=-\mathbb{E}[-Y]$, i.e., }$Y$ has no mean-uncertainty.{ Then we
have}%
\[
\mathbb{E}[X+\alpha Y]=\mathbb{E}[X]+\alpha \mathbb{{E}}[Y]\  \  \  \text{for}%
\  \alpha \in \mathbb{R}.
\]
{In particular, if $\mathbb{E}[Y]=\mathbb{E}[-Y]=0$, then $\mathbb{E}[X+\alpha
Y]=\mathbb{E}[X]$. }\bigskip
\end{proposition}

\begin{definition}
A sequence of $n$-dimensional random vectors $\left \{  \eta_{i}\right \}
_{i=1}^{\infty}$ defined on a nonlinear expectation space $(\Omega
,\mathcal{H},\mathbb{E})$ is said to \textbf{converge in distribution} (or
\textbf{converge in law}) under $\mathbb{E}$ if for each $\varphi \in
C_{b.Lip}(\mathbb{R}^{n})$, the sequence $\left \{  \mathbb{E}[\varphi(\eta
_{i})]\right \}  _{i=1}^{\infty}$ converges.
\end{definition}

The following result is easy to check.

\begin{proposition}
Let $\left \{  \eta_{i}\right \}  _{i=1}^{\infty}$ converges in law in the above
sense. Then the mapping $\mathbb{{F}}[\cdot]:C_{b.Lip}(\mathbb{R}%
^{n})\rightarrow \mathbb{R}$ defined by%
\[
\mathbb{{F}}[\varphi]:=\lim_{i\rightarrow \infty}\mathbb{E}[\varphi(\eta
_{i})]\  \  \text{for}\  \varphi \in C_{b.Lip}(\mathbb{R}^{n})
\]
is a nonlinear expectation defined on $(\mathbb{R}^{n},C_{b.Lip}%
(\mathbb{R}^{n}))$. If $\mathbb{E}$ is sublinear (resp. linear), then
$\mathbb{F}$ is also sublinear (resp. linear).
\end{proposition}

The following notion of independence plays a key role in the sublinear
expectation theory.

\begin{definition}
\label{d2} In a nonlinear expectation space $(\Omega,\mathcal{H},\mathbb{E})$,
a random vector $Y\in \mathcal{H}^{n}$ is said to be \textbf{independent} of
another random vector $X\in \mathcal{H}^{m}$ under $\mathbb{E}[\cdot]$ if for
each test function $\varphi \in C_{l.Lip}(\mathbb{R}^{m+n})$ we have
\[
\mathbb{E}[\varphi(X,Y)]=\mathbb{E}[\mathbb{E}[\varphi(x,Y)]_{x=X}].
\]

\end{definition}

\begin{remark}
In a sublinear expectation space $(\Omega,\mathcal{H},\mathbb{E})$, $Y$ is
independent of $X$ means that the uncertainty of distributions $\{F_{Y}%
(\theta,\cdot):\theta \in \Theta \}$ of $Y$ does not change after the realization
of $X=x$. In other words, the \textquotedblleft conditional sublinear
expectation\textquotedblright \ of $Y$ with respect to $X$ is $\mathbb{E}%
[\varphi(x,Y)]_{x=X}$. In the case of linear expectation, this notion of
independence is just the classical one.
\end{remark}

\begin{remark}
It is important to note that under a sublinear expectation the condition
\textquotedblleft$Y$ is independent from $X$\textquotedblright \ does not imply
automatically that \textquotedblleft$X$ is independent from $Y$%
\textquotedblright.
\end{remark}

The independence property of two random vectors $X,Y$ involves only the
``joint distribution'' of $(X,Y)$. The following result tells us how to
construct random vectors with given ``marginal distributions'' and with a
specific direction of independence.

\begin{definition}
\label{dd1} Let $(\Omega_{i},\mathcal{H}_{i},\mathbb{E}_{i})_{i\in I}$ be
nonlinear expectation spaces indexed by $I$ We denote%
\[%
{\displaystyle \prod \limits_{i\in I}}
\Omega_{i}=\{(\omega_{i}:i\in I):\omega_{i}\in \Omega_{i},\ i\in I\}
\]
\begin{align*}
\bigotimes_{i\in I}\mathcal{H}_{i}  &  :=\{Y(\omega_{i_{1}},\cdots
,\omega_{i_{n}})=\varphi(X_{i_{1}}(\omega_{i_{1}}),\cdots,X_{i_{n}}%
(\omega_{i_{n}})):(\omega_{1},\omega_{2})\in \Omega_{1}\times \Omega_{2},\  \\
\  \  &  \ X_{i_{k}}\in \mathcal{H}_{i_{k}}^{d_{k}},\ k=1,\cdots,n,\  \varphi \in
C_{l.Lip}(\mathbb{R}^{d_{1}+\cdots+d_{n}})\},\  \
\end{align*}

\end{definition}

We denote by $\Omega:=%
{\displaystyle \prod \limits_{i\in I}}
\Omega_{i}$ and $\mathcal{H}:=\bigotimes_{i\in I}\mathcal{H}_{i}$. It is clear
that $\mathcal{H}$ forms a vector lattice on $\Omega$. Let $\mathbb{E}$ be a
nonlinear expectation defined on $\mathcal{H}$. If for each $i\in I$ and
$X_{i}\in \mathcal{H}_{i}^{d_{i}}$ we always have $\mathbb{E}[\varphi
(X_{i})]=\mathbb{E}_{i}[\varphi(X_{i})]$, then we say that the margin of
$\mathbb{E}$ coincides with $\mathbb{E}_{i}$.

\begin{remark}
In the last section, the above notion of independence is extended to an
$\mathbb{C}$-valued linear expectation space $(\Omega,\mathcal{H},\mathbb{E})$.
\end{remark}

\subsection{Completion of a sublinear expectation space}

\label{c1s5} Let $(\Omega,\mathcal{H},\mathbb{E})$ be a sublinear expectation
space. We have the following useful inequalities. { }

We first give the following well-known inequalities.

\begin{lemma}
{ {F{or $r>0$ and $1<p,q<\infty$ with $\frac{1}{p}+\frac{1}{q}=1$, we have
\begin{align}
|a+b|^{r}  &  \leq \max \{1,2^{r-1}\}(|a|^{r}+|b|^{r})\  \  \text{for}%
\ a,b\in \mathbb{R},\label{ee04.3}\\
|ab|  &  \leq \frac{|a|^{p}}{p}+\frac{|b|^{q}}{q}. \label{ee04.4}%
\end{align}
} } }
\end{lemma}

\begin{proposition}
\label{ppp1} { { For each $X,Y\in$}}$\mathcal{H}${{, we have{
\begin{align}
\mathbb{E}[|X+Y|^{r}]  &  \leq2^{r-1}(\mathbb{E}[|X|^{r}]+\mathbb{{E}[}%
|Y|^{r}]),\label{ee04.5}\\
\mathbb{E}[|XY|]  &  \leq(\mathbb{E}[|X|^{p}])^{1/p}\cdot(\mathbb{E}%
[|Y|^{q}])^{1/q},\label{ee04.6}\\
(\mathbb{E}[|X+Y|^{p}])^{1/p}  &  \leq(\mathbb{E}[|X|^{p}])^{1/p}%
+(\mathbb{E}[|Y|^{p}])^{1/p}, \label{ee04.7}%
\end{align}
where {{$r$}}}}}${{{{{\geq1}}}}}${{{{{ and $1<p,q<\infty$ with $\frac{1}%
{p}+\frac{1}{q}=1.$}} }}}

In particular, for $1\leq p<p^{\prime}$, we have $(\mathbb{E}[|X|^{p}%
])^{1/p}\leq(\mathbb{E}[|X|^{p^{\prime}}])^{1/p^{\prime}}.$
\end{proposition}

\begin{proof}
The inequality (\ref{ee04.5}) follows from (\ref{ee04.3}).

For the case $\mathbb{E}[|X|^{p}]\cdot \mathbb{E}[|Y|^{q}]>0$, we set
\[
\xi=\frac{X}{(\mathbb{E}[|X|^{p}])^{1/p}},\  \  \eta=\frac{Y}{(\mathbb{E}%
[|Y|^{q}])^{1/q}}.
\]
By (\ref{ee04.4}) we have%
\begin{align*}
\mathbb{E}[|\xi \eta|]  &  \leq \mathbb{E}[\frac{|\xi|^{p}}{p}+\frac{|\eta|^{q}%
}{q}]\leq \mathbb{E}[\frac{|\xi|^{p}}{p}]+\mathbb{{E}}[\frac{|\eta|^{q}}{q}]\\
&  =\frac{1}{p}+\frac{1}{q}=1.
\end{align*}
Thus (\ref{ee04.6}) follows.

For the case $\mathbb{E}[|X|^{p}]\cdot \mathbb{E}[|Y|^{q}]=0,$ we consider
$\mathbb{E}[|X|^{p}]+\varepsilon$ and $\mathbb{E}[|Y|^{q}]+\varepsilon$ for
$\varepsilon>0.$ Applying the above method and letting $\varepsilon
\rightarrow0,$ we get (\ref{ee04.6}).

We now prove (\ref{ee04.7}). We only consider the case ${\mathbb{E}}%
[|X+Y|^{p}]>0$.
\begin{align*}
\mathbb{E}[|X+Y|^{p}]  &  =\mathbb{E}[|X+Y|\cdot|X+Y|^{p-1}]\\
&  \leq \mathbb{E}[|X|\cdot|X+Y|^{p-1}]+\mathbb{E}[|Y|\cdot|X+Y|^{p-1}]\\
&  \leq(\mathbb{E}[|X|^{p}])^{1/p}\cdot(\mathbb{{E}[}|X+Y|^{(p-1)q}])^{1/q}\\
&  \  \  \ +(\mathbb{E}[|Y|^{p}])^{1/p}\cdot(\mathbb{{E}[}|X+Y|^{(p-1)q}%
])^{1/q}.
\end{align*}
Since $(p-1)q=p$, we have (\ref{ee04.7}).

{By}(\ref{ee04.6}), it is easy to deduce that {{{$(\mathbb{E}[|X|^{p}%
])^{1/p}\leq(\mathbb{E}[|X|^{p^{\prime}}])^{1/p^{\prime}}$ for $1\leq
p<p^{\prime}.$}}}
\end{proof}

\bigskip

{ {For each fixed }$p\geq1$, we observe that $\mathcal{H}_{0}^{p}%
=\{X\in \mathcal{H}$, $\mathbb{E}[|X|^{p}]=0\}$ is a linear subspace of
$\mathcal{H}$. Taking $\mathcal{H}_{0}^{p}$ as our null space, we introduce
the quotient space $\mathcal{H}/\mathcal{H}_{0}^{p}$. Observing that, for
every $\mathbf{\{}X\} \in \mathcal{H}/\mathcal{H}_{0}^{p}$ with a
representation $X\in \mathcal{H}$, we can define an expectation $\mathbb{E}%
\mathbf{[\{}X\}]:=\mathbb{E}[X]$ which is still a sublinear expectation. We
set $\left \Vert X\right \Vert _{p}:=(\mathbb{E}[|X|^{p}])^{\frac{1}{p}}$. By
Proposition \ref{ppp1}, i{{t is easy to check that $\left \Vert \cdot
\right \Vert _{p}$ forms a Banach norm on {{$\mathcal{H}/\mathcal{H}_{0}^{p}$}%
}. We extend $\mathcal{H}/\mathcal{H}_{0}^{p}$ to its completion
$\mathcal{\hat{H}}_{p}$ under\ this norm, then $(\mathcal{\hat{H}}%
_{p},\left \Vert \cdot \right \Vert _{p})$ is a Banach space. \ In particular,
when }}}$p=1,$ we denote it by {{{$(\mathcal{\hat{H}},\left \Vert
\cdot \right \Vert ).$}}}

{{{For each $X\in \mathcal{H}$, the mappings
\[
X^{+}(\omega):\mathcal{H\rightarrow H}\  \  \  \text{and \  \ }X^{-}%
(\omega):\mathcal{H\rightarrow H}%
\]
satisfy
\[
|X^{+}-Y^{+}|\leq|X-Y|\text{ \  \ and \ }\ |X^{-}-Y^{-}|=|(-X)^{+}%
-(-Y)^{+}|\leq|X-Y|.
\]
Thus they are both contraction mappings under $\left \Vert \cdot \right \Vert
_{p}$ and can be continuously extended to the Banach space $(\mathcal{\hat{H}%
}_{p},\left \Vert \cdot \right \Vert _{p})$. } } }

{ { { We can define the partial order \textquotedblleft$\geq$%
\textquotedblright \ in this Banach space. }}}

\begin{definition}
{ { { An element $X$ in $(\mathcal{\hat{H}},\left \Vert \cdot \right \Vert )$ is
said to be nonnegative, or $X\geq0$, $0\leq X$, if $X=X^{+}$. We also denote
by $X\geq Y$, or $Y\leq X$, if $X-Y\geq0$. } } }
\end{definition}

{ { { It is easy to check that $X\geq Y$ and $Y\geq X$ imply $X=Y$ on
$(\mathcal{\hat{H}}_{p},\left \Vert \cdot \right \Vert _{p})$. }}}

For each {{{$X,Y\in \mathcal{H}$, note that }}}%

\[
|{{{\mathbb{E}[X]-\mathbb{E}[Y]|\leq \mathbb{E}[|X-Y|]\leq||X-Y||}}}_{p}.
\]
Thus the {{{sublinear expectation $\mathbb{E}[\cdot]$ can be continuously
extended to $(\mathcal{\hat{H}}_{p},\left \Vert \cdot \right \Vert _{p})$ on
which it is still a sublinear expectation.}}}

Let $(\Omega,\mathcal{H},\mathbb{E}_{1})$ be a nonlinear expectation space.
$\mathbb{E}_{1}$ is said to be \ dominated by $\mathbb{E}$ if%

\[
\mathbb{E}_{1}[X]-\mathbb{E}_{1}[Y]\leq \mathbb{{E}}[X-Y]\  \  \  \  \text{for}%
\ X,Y\in \mathcal{H}.
\]
From this we can easily deduce that $|{{{\mathbb{E}}}}_{1}{{{[X]-\mathbb{E}}}%
}_{1}{{{[Y]|\leq \mathbb{E}[|X-Y|],}}}$ thus the nonlinear {{{expectation
$\mathbb{E}_{1}[\cdot]$ can be continuously extended to $(\mathcal{\hat{H}%
}_{p},\left \Vert \cdot \right \Vert _{p})$ on which it is still a nonlinear
expectation.}}}

\begin{remark}
It is important to note that $X_{1},\cdots,X_{n}\in \mathcal{\hat{H}}$ does not
imply $\varphi(X_{1},\cdots,X_{n})\in \mathcal{\hat{H}}$ for each $\varphi \in
C_{l.Lip}(\mathbb{R}^{n}).$ Thus, when we talk about the notions of
distributions, independence and product spaces on $(\Omega,\mathcal{\hat{H}%
},\mathbb{E}),$ the space $C_{l.Lip}(\mathbb{R}^{n})$ is replaced by
$C_{b.Lip}(\mathbb{R}^{n})$ unless otherwise stated.
\end{remark}

\subsection{$G$-normal distributions}

A well-known characterization for a zero-mean $d$-dimensional normally
distributed random variable $X$ is%

\begin{equation}
aX+b\bar{X}\overset{d}{=}\sqrt{a^{2}+b^{2}}X\  \  \text{for }a,b\geq0,
\label{ch2e1}%
\end{equation}
where $\bar{X}$ is an independent copy of $X$. The covariance matrix $\Sigma$
is defined by $\Sigma=E[XX^{T}]$. We now consider the so called $G$-normal
distribution in probability model uncertainty situation.

\begin{definition}
(\textbf{$G$-normal distribution}) A $d$-dimensional random vector
$X=(X_{1},\cdots,X_{d})$ on a sublinear expectation space $(\Omega
,\mathcal{H},\mathbb{E})$ is called (centralized) \textbf{$G$-normal
distributed} if
\begin{equation}
aX+b\bar{X}\overset{d}{=}\sqrt{a^{2}+b^{2}}X\  \  \  \text{for }a,b\geq
0,\  \label{e311}%
\end{equation}
where $\bar{X}$ is an independent copy of $X$.
\end{definition}

\begin{remark}
Noting that $\mathbb{E}[X+\bar{X}]=2\mathbb{E}[X]$ and $\mathbb{E}[X+\bar
{X}]=\mathbb{E}[\sqrt{2}X]=\sqrt{2}\mathbb{E}[X]$, we then have $\mathbb{E}%
[X]=0.$ Similarly, we can prove that $\mathbb{E}[-X]=0.$ Namely, $X$ has no mean-uncertainty.
\end{remark}

The following property is easy to prove by the definition.

\begin{proposition}
\label{GCD}Let $X$ be $G$-normal distributed. Then for each $A\in
\mathbb{R}^{m\times d}$, $AX$ is also $G$-normal distributed. In particular,
for each $\mathbf{a}\in \mathbb{R}^{d},$ $\left \langle \mathbf{a}%
,X\right \rangle $ is a $1$-dimensional $G$-normal distributed random variable.
\end{proposition}

We denote by $\mathbb{S}(d)$ the collection of all $d\times d$ symmetric
matrices. Let $X$ be a $d$-dimensional $G$-normal distributed random vector in
$(\Omega,\mathcal{H},\mathbb{E})$. The following function is very important to
characterize their distributions:
\begin{equation}
G(A):=\mathbb{E}[\frac{1}{2}\left \langle AX,X\right \rangle ],\  \  \ A\in
\mathbb{S}(d). \label{e313}%
\end{equation}
It is easy to check that $G$ is a sublinear function monotonic in
$A\in \mathbb{S}(d)$ in the following sense: for each $A,\bar{A}\in
\mathbb{S}(d)$%
\begin{equation}
\left \{
\begin{array}
[c]{rl}%
G(A+\bar{A}) & \leq G(A)+G(\bar{A}),\\
G(\lambda A) & =\lambda G(A),\  \  \forall \lambda \geq0,\\
G(A) & \geq G(\bar{A}),\  \  \text{if}\ A\geq \bar{A}.
\end{array}
\right.  \label{e314}%
\end{equation}
Clearly, $G$ is also a continuous function. By Theorem \ref{t1}, there exists
a bounded and closed subset $\Gamma \subset \mathbb{R}^{d\times d}$ such that
\begin{equation}
G(A)=\sup_{Q\in \Gamma}\frac{1}{2}\mathrm{tr}[AQQ^{T}]\  \  \  \text{for}%
\ A\in \mathbb{S}(d). \label{ch2e2}%
\end{equation}

The following result can be found in [Peng2010].

\begin{proposition}
\label{Prop-Gnorm copy(1)} Let $G:\mathbb{S}(d)\rightarrow \mathbb{R}$ be a
given sublinear and continuous function, monotonic in $A\in \mathbb{S}(d)$ in
the sense of \textup{(\ref{e314})}. Then there exists a $G$-normal distributed
$d$-dimensional random vector $X$ on some sublinear expectation space
$(\Omega,\mathcal{H},\mathbb{E})$. satisfying \textup{(\ref{e311})}. Moreover,
if both $X$ and $Y$ are $G$-normal \ distributed with the same function $G$,
namely%
\[
\mathbb{E}[\left \langle AX,X\right \rangle ]=\mathbb{E}[\left \langle
AY,Y\right \rangle ]=2G(A),\  \  \forall A\in \mathbb{S}(d),
\]
then $X\overset{d}{=}Y$.
\end{proposition}

We present a central limit theorem in the framework of sublinear expectation
(see [Peng2007], [Peng2009] or [Peng2010, ThmII.3.3).

\begin{theorem}
\label{tCLT}Let $\{X_{i}\}_{i=1}^{\infty}$ be a sequence of $\mathbb{R}^{d}$
valued random vectors in a sublinear expectation space $(\Omega,\mathcal{H}%
,\mathbb{E})$. We assume that $\{X_{i}\}_{i=1}^{\infty}$ is an i.i.d.
sequence, i.e., for each $i=1,2,\cdots$, $X_{i+1}\overset{d}{=}X_{1}$ and it
is independent of $(X_{1},\cdots,X_{i})$. We assume furthermore that
$\mathbb{E}[X_{1}]=\mathbb{E}[-X_{1}]=0$. Then the sequence $\{ \bar{S}%
_{n}\}_{n=1}^{\infty}$ defined by $\bar{S}_{n}=(X_{1}+\cdots+X_{n})/\sqrt{n}$
converges in law to $X$:%
\[
\lim_{n\rightarrow \infty}\mathbb{E}[\varphi(\bar{S}_{n})]=\mathbb{E}%
[\varphi(X)],\  \  \varphi \in C_{b}(\mathbb{R}^{d}),
\]
where $X$ is a d-dimensional $G$-normally distributed random variable with
\[
G(A)=\frac{1}{2}\mathbb{E}[\left \langle AX_{1},X_{1}\right \rangle ].
\]

\end{theorem}

\section{Construction of stochastic processes with a given family of finite
dimensional distributions}

\begin{definition}
Let $(\Omega,\mathcal{H},\mathbb{{E})}$ be a nonlinear expectation space.
$(X_{t})_{t\geq0}$ is called a $d$-dimensional \textbf{stochastic process} if
for each $t\geq0$, $X_{t}$ is a $d$-dimensional random vector in $\mathcal{H}$.
\end{definition}

A typical example of such type of stochastic processes defined on a space of
nonlinear expectation is the so-called $G$-Browian motion.

\begin{definition}
A $d$-dimensional process $(B_{t})_{t\geq0}$ on a sublinear expectation space
$(\Omega,\mathcal{H},\mathbb{E})$ is called a $G$\textbf{--Brownian}
\textbf{motion} if the following properties are satisfied: \newline%
\textup{(i)} $B_{0}(\omega)=0$;\newline \textup{(ii)} For each $t,s\geq0$, the
increment $B_{t+s}-B_{t}\overset{d}{=}B_{s}$ and is independent of $(B_{t_{1}%
},B_{t_{2}},\cdots,B_{t_{n}})$, for each $n\in \mathbb{N}$ and $t_{1}%
,\cdots,t_{n}\in \lbrack0,t]$.\newline(iii) $\lim_{t\rightarrow0}%
\mathbb{E}[|B_{t}|^{3}]/t=0.$
\end{definition}

The following theorem gives a characterization of $G$-Brownian motion.

\begin{theorem}
Let $(B_{t})_{t\geq0}$ be a $d$-dimensional $G$-Brownian motion defined on a
sublinear expectation space $(\Omega,\mathcal{H},\mathbb{E})${, such that
}$\mathbb{E}${$[B_{t}]=\mathbb{E}[-B_{t}]=0$, }$t\geq0$. Then, for each $t>0$,
$B_{t}${ is normally distributed, namely,}%
\[
aB_{t}+b(B_{2t}-B_{t})\overset{d}{=}\sqrt{a^{2}+b^{2}}B_{t},\ a,b\  \geq0.
\]
{\newline}
\end{theorem}

Using a generalization of Kolmogorov approach, a new type of Markovian
processes was introduced in [Peng2004] and then $G$-Brownian motion in [Peng2007].

In this section we will use this approach to give a general construction of
stochastic processes so that it can be also applied to construct a new type of
Gaussian processes under a sublinear expectation.

We first notice that, just as in the classical situation, one can define the
family of finite dimensional distributions for a given stochastic process
$X=(X_{t})_{t\geq0}$. We denote%

\[
\mathcal{T}:=\{ \underline{t}=(t_{1},\ldots,t_{m}):\forall m\in \mathbb{N}%
,t_{i}\in[0,\infty),\, t_{i}\not =t_{j}\, \, 0\leq i,j\leq m, i\not =j \}.
\]

\begin{definition}
\label{Def3.4}If for, each $\underline{t}=(t_{1},\ldots,t_{m})\in \mathcal{T}$,
$\mathbb{F}_{\underline{t}}$ is a nonlinear expectation defined on
$(\mathbb{R}^{m\times d},C_{l.Lip}(\mathbb{R}^{m\times d}))$ then we call
$(\mathbb{F}_{\underline{t}})_{\underline{t}\in \mathcal{T}}$ a family of
finite dimensional distributions on $\mathbb{R}^{d}$.
\end{definition}

\begin{definition}
Let $(X_{t})_{t\geq0}$ be an $\mathbb{R}^{d}$-valued stochastic process
defined in a nonlinear expectation space $(\Omega,\mathcal{H},\mathbb{E})$.
For each $\underline{t}=(t_{1},\cdots,t_{n})\in \mathcal{T}$ the random
variable $X_{\underline{t}}=(X_{t_{1}},\cdots,X_{t_{n}})$ induces a
distribution distribution $\mathbb{F}_{\underline{t}}^{X}[\varphi
]=\mathbb{E}[\varphi(X_{\underline{t}})]$, $\varphi \in C_{l.Lip}%
(\mathbb{R}^{n\times d})$. We call $(\mathbb{F}_{\underline{t}}^{X}%
)_{\underline{t}\in \mathcal{T}}$ the family of finite dimensional
distributions corresponding to $(X_{t})_{t\geq0}$.
\end{definition}

\begin{definition}
Let $(X_{t})_{t\geq0}$ and $(\bar{X}_{t})_{t\geq0}$ be $d$-dimensional
\textbf{stochastic processes }defined respectively in nonlinear expectation
spaces $(\Omega,\mathcal{H},\mathbb{E})$ and $(\bar{\Omega},\mathcal{\bar{H}%
},\mathbb{{\bar{E}})}$. $X$ and $\bar{X}$ are said to be indentically
distributed if their families of finite dimensional distributions coincide
from each other:%
\[
\mathbb{F}_{\underline{t}}^{X}=\mathbb{F}_{\underline{t}}^{\bar{X}%
},\  \  \forall \underline{t}\in \mathcal{T}.
\]

\end{definition}

For a given stochastic process, the family of its finite dimensional
distributions satisfies the following properties of consistency:

\begin{definition}
\label{Def3.7}A family of finite dimentional distributions $\{ \mathbb{F}%
_{\underline{t}}[\varphi]\}_{\underline{t}\in \mathcal{T}}$ on $\mathbb{R}^{d}$
is said to be consistent if, for each $\underline{t}=(t_{1},\cdots,t_{n}%
)\in \mathcal{T}$, we have\newline(i)
\[
\mathbb{F}_{\underline{t}}[\varphi]=\mathbb{F}_{\underline{t}}[\varphi
],\  \  \varphi \in C_{l.Lip}(\mathbb{R}^{n\times d}).
\]
Here, on the right hand side, $\varphi$ is considered as a funtion defined on
$\mathbb{R}^{n\times d}\times \mathbb{R}^{d}$ which does not depend on the last
coordinate. \newline(ii) For each permutation $\sigma$ of $(1,2,\cdots,n)$
\[
\mathbb{F}_{t_{\sigma(1)},\cdots,t_{\sigma(n)}}[\varphi]=\mathbb{F}%
_{t_{1},\cdots,t_{n}}[\varphi_{\sigma}]
\]
where%
\[
\varphi_{\sigma}(x_{1},\cdots,x_{n})=\varphi(x_{\sigma(1)},\cdots
,x_{\sigma(n)}),\  \ x_{i}\in \mathbb{R}^{d},\  \ i=1,\cdots,n.
\]

\end{definition}

It is clear that the finite dimensional distributions of the process $X$ is
consistent. Inversely, we have:

\begin{theorem}
\label{Thm8.8}Let $(\mathbb{F}_{\underline{t}})_{\underline{t}\in \mathcal{T}}$
be a family of consistent nonlinear distributions. Then there exists a
$d$-dimensional stochactic process \ $(X_{t})_{t\geq0}$ defined on a nonlinear
expectation space $(\Omega,\mathcal{H},\mathbb{E})$ such that the family of
finite dimensional distributions of $X$ coincides with $(\mathbb{F}%
_{\underline{t}})_{\underline{t}\in \mathcal{T}}$. $\mathbb{E}$ can be a
sublinear (resp. linear) expectation if the distributions in $(\mathbb{F}%
_{\underline{t}})_{\underline{t}\in \mathcal{T}}$ are all sublinear (resp. linear).
\end{theorem}

\begin{proof}
Let$\  \Omega=(\mathbb{R}^{d})^{[0,\infty)}$ denote the space of all
$\mathbb{R}^{d}$-valued functions $(\omega_{t})_{t\in \mathbb{R}^{+}}$. We
denote by $\overline{X}_{t}(\omega)=\omega_{t}$, $t\in \lbrack0,\infty)$,
$\omega \in \Omega$, the corresponding canonical process. The space of
Lipschitzian cylinder functions on $\Omega$ is denoted by%
\[
\mathcal{H}=L_{ip}(\Omega):=\{ \varphi(X_{t_{1}},\cdots,X_{t_{n}}%
),\underline{t}=(t_{1},\cdots,t_{n})\in \mathcal{T},\forall \varphi \in
C_{l.Lip}(\mathbb{R}^{d\times n})\}.
\]

It is clear that $L_{ip}(\Omega)$ is a vector lattice. For each $\xi \in
L_{ip}(\Omega)$ of the form $\xi(\omega)=\varphi(X_{t_{1}}(\omega
),\cdots,X_{t_{n}}(\omega))$, we set
\[
\mathbb{E}[\varphi(\xi)]=\mathbb{F}_{t_{1},\cdots,t_{n}}[\varphi].
\]
It is clear that the mapping $\mathbb{E}:L_{ip}(\Omega)\mapsto \mathbb{R}$
forms a consistently defined nonlinear expectation on $(\Omega,L_{ip}%
(\Omega))$ and the family of the finite dimensional distributions of $X$ is
$(\mathbb{F}_{\underline{t}})_{\underline{t}\in \mathcal{T}}$.

We see that with this construction $\mathbb{E}$ is sublinear (resp. linear)
expectation if the distributions in $(\mathbb{F}_{\underline{t}}%
)_{\underline{t}\in \mathcal{T}}$ are all sublinear (resp. linear).
\end{proof}

\begin{remark}
In the proof of Theorem \ref{Thm8.8}, we can also use $\Omega=C(0,\infty
;\mathbb{R}^{d})$ in the place of $\bar{\Omega}=(\mathbb{R}^{d})^{[0,\infty)}%
$. But we will need the later one in the proof of Lemma \ref{Lemma3.12}. In
many situations, similar to the classical situation, we need to introduce the
natural capacity $\hat{c}$ associated to $\mathbb{E}$ and use the related
\textquotedblleft$\hat{c}$-`quasi sure' analysis\textquotedblright \ to study
the continuity of the process $X$. We refer to [Denis-Hu-Peng2010] or
[Hu-Peng2010] for the proof of the continuity of a $G$-Brownian motion.
\end{remark}

\begin{remark}
Definitions \ref{Def3.4}-\ref{Def3.7} as well as Theorem \ref{Thm8.8} can be
extended to construct a $\mathbb{R}^{d}$-valued nonlinear expectation. We will
see in Section 6 a typical $\mathbb{C}$-valued expectation.
\end{remark}

We will prove that when $(\mathbb{F}_{\underline{t}})_{\underline{t}%
\in \mathcal{T}}$ is sublinear, then the corresponding sublinear expectation
$\mathbb{E}$ constructed above is an upper expectation of a family of
probability measures on $(\bar{\Omega},\bar{\mathcal{F}})$. We need the
following lemmas:

\begin{lemma}
\label{le3}Let $\xi \in \mathcal{H}^{m}$ be a given random vector in a linear
expectation space $(\Omega,\mathcal{H},E)$ such that $\varphi(\xi
)\in \mathcal{H}$ for each $\varphi \in C_{l.Lip}(\mathbb{R}^{m})$. Then there
exists a unique probability measure $Q$ on $(\mathbb{R}^{m},\mathcal{B}%
(\mathbb{R}^{m}))$ such that $E_{Q}[\varphi]=E[\varphi(\xi)]$, $\varphi \in
C_{l.Lip}(\mathbb{R}^{m})$.
\end{lemma}

\begin{proof}
Let $\{ \varphi_{n}\}_{n=1}^{\infty}$ be a sequence in $C_{l.Lip}%
(\mathbb{R}^{m})$ satisfying $\varphi_{n}\downarrow0$. For each $N>0$, it is
clear that
\[
\varphi_{n}(x)\leq k_{n}^{N}+\varphi_{1}(x)I_{[|x|>N]}\leq k_{n}^{N}%
+\frac{\varphi_{1}(x)|x|}{N}\text{\ for each }x\in \mathbb{R}^{d\times m},
\]
where $k_{n}^{N}=\max_{|x|\leq N}\varphi_{n}(x)$. Noting that $\varphi
_{1}(x)|x|\in C_{l.Lip}(\mathbb{R}^{m})$, we have
\[
\mathbb{E}[\varphi_{n}(\xi)]\leq k_{n}^{N}+\frac{1}{N}\mathbb{E}[\varphi
_{1}(\xi)|\xi|].
\]
It follows from $\varphi_{n}\downarrow0$ that $k_{n}^{N}\downarrow0$. Thus we
have $\lim_{n\rightarrow \infty}\mathbb{E}[\varphi_{n}(\xi)]\leq \frac{1}%
{N}\mathbb{E}[\varphi_{1}(\xi)|\xi|]$. Since $N$ can be arbitrarily large, we
get $\mathbb{E}[\varphi_{n}(\xi)]\downarrow0$. Consequently, $E[\varphi
_{n}(\xi)]\downarrow0$.

It follows from Daniell-Stone's theorem that there exists a unique probability
measure $Q$ on $(\mathbb{R}^{m},\mathcal{B}(\mathbb{R}^{m}))$ such that
$E_{Q}[\varphi]=E[\varphi(X)]$, for each $\varphi \in C_{l.Lip}(\mathbb{R}%
^{m})$.
\end{proof}

\bigskip

\begin{lemma}
\label{Lemma3.12} \label{le4 copy(1)} Let $\mathbb{E}$ be a sublinear
expectation on $(\Omega,L_{ip}(\Omega))$ and $E$ be a (finitely additive)
linear expectation on $(\Omega,L_{ip}(\Omega))$ which is dominated by
$\mathbb{E}.$ Then there exists a unique probability measure $Q$ on
$(\Omega,\sigma(L_{ip}(\Omega))$ such that $E[X]=E_{Q}[\xi]$ for each $\xi \in
L_{ip}(\Omega)$.
\end{lemma}

\begin{proof}
For each fixed $\underline{t}=(t_{1},\ldots,t_{m})\in \mathcal{T}$, we denote
$X=(X_{t_{1}},\cdots,X_{t_{m}})$. by Lemma \ref{le3} there exists a unique
probability measure $Q_{\underline{t}}$ on $(\mathbb{R}^{d\times
m},\mathcal{B}(\mathbb{R}^{d\times m}))$ such that $E_{Q_{\underline{t}}%
}[\varphi]=E[\varphi(X_{t_{1}},\cdots,X_{t_{m}})]$ for each $\varphi \in
C_{l.Lip}(\mathbb{R}^{d\times m})$. Thus, corresponding to $(X_{t_{1}}%
,\cdots,X_{t_{m}})$, $\underline{t}=(t_{1},\ldots,t_{m})\in \mathcal{T}$, we
get a family of finite dimensional distributions $\{Q_{\underline{t}%
}:\underline{t}\in \mathcal{T}\}$. It is easy to check that $\{Q_{\underline
{t}}:\underline{t}\in \mathcal{T}\}$ is consistent. Then by Kolmogorov's
consistent theorem, there exists a probability measure $Q$ on $(\Omega
,\sigma(L_{ip}(\Omega))$ such that $\{Q_{\underline{t}}:\underline{t}%
\in \mathcal{T}\}$ is the finite dimensional distributions of $Q$. Now assume
that there exists another probability measure $\bar{Q}$ satisfying the
condition, by Daniell-Stone's theorem, $Q$ and $\bar{Q}$ have the same
finite-dimensional distributions. Then by monotone class theorem, $Q=\bar{Q}$.
The proof is complete.
\end{proof}

\begin{lemma}
\label{le5} There exists a family of probability measures $\mathcal{P}_{e}$ on
$(\Omega,\sigma(\Omega))$ such that
\[
\mathbb{\bar{E}}[X]=\max_{Q\in \mathcal{P}_{e}}E_{Q}[X],\quad \text{for}\ X\in
L_{ip}(\Omega).
\]

\end{lemma}

\begin{proof}
By the representation theorem of sublinear expectation and Lemma
\ref{Lemma3.12}, it is easy to get the result.
\end{proof}

\  \  \

For this $\mathcal{P}_{e}$, we define the associated capacity:
\[
\tilde{c}(A):=\sup_{Q\in \mathcal{P}_{\! \!e}}Q(A),\quad A\in \mathcal{B}%
(\bar{\Omega}),
\]
and the upper expectation for each $\mathcal{B}(\bar{\Omega})$-measurable real
function $X$ which makes the following definition meaningful:
\[
\mathbb{\tilde{E}}[X]:=\sup_{Q\in \mathcal{P}_{\! \!e}}E_{Q}[X].
\]

\section{Gaussian processes in a sublinear expectation space}

In this section we generalize the notion of Gaussian processes to the
situation in sublinear expectation space.

\begin{definition}
An $\mathbb{R}^{d}$-valued stochastic process $X=(X_{t})_{t\geq0}$ defined in
a sublinear expectation space $(\Omega,\mathcal{H},\mathbb{E})$ is called a
Gaussian process if for each $\underline{t}=(t_{1},\cdots,t_{n})\in
\mathcal{T}$, $X_{\underline{t}}=(X_{t_{1}},\cdots,X_{t_{n}})$ is an
$\mathbb{R}^{n\times d}$-valued normally distributed random variable.
\end{definition}

In this section we are only concerned with the processes satisfying
$\mathbb{E}[X_{t}]=\mathbb{E}[-X_{t}]=0$.

\begin{definition}
Let $(X_{t})_{t\geq0}$ be a given $d$-dimensional stochastic process. We
denote, for each $\underline{t}=(t_{1},\cdots,t_{n})\in \mathcal{T}$,%
\[
G_{\underline{t}}^{X}(A):=\frac{1}{2}\mathbb{E[}\left \langle AX_{\underline
{t}},X_{\underline{t}}\right \rangle ]:A\in \mathbb{S}_{n\times d}.
\]
We called $(G_{\underline{t}}^{X})_{\underline{t}\in \mathcal{T}}$ the family
of 2nd moments of the process $X$. It is clear that for each $\underline
{t}=(t_{1},\cdots,t_{n})\in \mathcal{T}$, $G_{\underline{t}}^{X}:\mathbb{S}%
_{n\times d}\mapsto \mathbb{R}$ is a sublinear and monotone function.
\end{definition}

\begin{definition}
A family $G_{t_{1},\cdots,t_{n}}:\mathbb{S}_{n\times d}\mapsto \mathbb{R}$,
$\underline{t}=(t_{1},\cdots,t_{n})\in \mathcal{T}$ of sublinear and monotone
functions is called consistent if it satisfies, for each $\underline{t}%
=(t_{1},\cdots,t_{n})\in \mathcal{T}$ and $t_{n+1}\geq0$, \newline(i)
$G_{t_{1},\cdots,t_{n+1}}(\overline{A})=G_{\underline{t}}(A)$, for each
$A\in \mathbb{S}_{n\times d}$ where
\[
\overline{A}=\left[
\begin{array}
[c]{cc}%
A & 0\\
0 & 0
\end{array}
\right]  \in \mathbb{S}_{(n+1)\times d}%
\]
(ii) $G_{t_{\sigma(1)},\cdots,t_{\sigma(n)}}(A)=G_{\underline{t}}(\sigma(A))$,
where, for each $A\in \mathbb{S}_{n\times d}$ with the form $A=[A_{ij}%
]_{i,j=1}^{n}$, $A_{ij}\in \mathbb{R}^{d\times d}$, $\sigma(A)$ is defined by
$\sigma(A)=[A_{\sigma(i)\sigma(j)}]_{i,j=1}^{n}$.
\end{definition}

It is clear that the family $(G_{\underline{t}}^{X})_{\underline{t}%
\in \mathcal{T}}$ of 2nd moments of the process $X$ is consistent. Inversly, we have:

\begin{proposition}
Let a family of sublinear monotone functions $\{G_{\underline{t}%
}\}_{\underline{t}\in \mathcal{T}}$ be consistent. Then for each $\underline
{t}=(t_{1},\cdots,t_{n})\in \mathcal{T}$, there is a unique $n\times
d$-dimensional normal distribution $\mathbb{F}_{\underline{t}}$ defined on
$(\mathbb{R}^{n\times d},C_{l.Lip}(\mathbb{R}^{n\times d}))$. Moreover the
family of finite distributions $\{ \mathbb{F}_{\underline{t}}\}_{\underline
{t}\in \mathcal{T}}$ is consistant. Consequently there exists a $d$-dimensional
Gaussian process $(X_{t})_{t\geq0}$ in a sublinear expectation space
$(\Omega,\mathcal{H},\mathbb{E})$ such that
\[
(G_{\underline{t}}^{X})_{\underline{t}\in \mathcal{T}}=(G_{\underline{t}%
})_{\underline{t}\in \mathcal{T}},\ (\mathbb{F}_{\underline{t}}^{X}%
)_{\underline{t}\in \mathcal{T}}=(\mathbb{F}_{\underline{t}})_{\underline{t}%
\in \mathcal{T}}\
\]
for each $\underline{t}=(t_{1},\cdots,t_{n})\in \mathcal{T}$, the random vector
$X_{\underline{t}}=(X_{t_{1}},\cdots,X_{t_{n}})$ is $G$-normal distributed
with $G=G_{\underline{t}}$.
\end{proposition}

\begin{example}
Let $(B_{t})_{t\geq0}$ be a $d$-dimensional $G$-Brownian motion in a sublinear
expectation space $(\Omega,\mathcal{H},\mathbb{E})$ with
\[
G(A)=\frac{1}{2}\mathbb{E}[\left \langle AB_{1},B_{1}\right \rangle ]
\]
For each $\underline{t}=(t_{1},\cdots,t_{n})\in \mathcal{T}$, we set
$B_{\underline{t}}=(B_{t_{1}},\cdots,B_{t_{n}})$ and%
\[
G_{\underline{t}}^{B}(A):=\frac{1}{2}\mathbb{E[}\left \langle AB_{\underline
{t}},B_{\underline{t}}\right \rangle ]:A\in \mathbb{S}_{n\times d}%
\mapsto \mathbb{R}.
\]
$\{G_{\underline{t}}^{B}\}_{\underline{t}\in \mathcal{T}}$ is the family of 2nd
moments of $(B)_{t\geq0}$ and thus satisfying the above consistency. We then
can construct a $G$-Gaussian process $(X_{t})_{t\geq0}$ such that,
$G_{\underline{t}}^{B}=G_{\underline{t}}^{X}$ for each $\underline{t}%
\in \mathcal{T}$. But, in general, their family of finite dimensional
distributions $(\mathbb{F}_{\underline{t}}^{B})_{\underline{t}\in \mathcal{T}}$
and $(\mathbb{F}_{\underline{t}}^{X})_{\underline{t}\in \mathcal{T}}$ are not
the same.

Since we still have
\[
\mathbb{E}[|X_{t}-X_{s}|^{4}]=\mathbb{E}[|B_{t}-B_{s}|^{4}]\leq d|t-s|^{2},
\]
We then can apply the same arguments as in the case of $G$-Brownian motion to
prove that there exists a weakly compact family $\mathcal{P}$ of probability
measures on $(\bar{\Omega},\mathcal{B}(\bar{\Omega}))$, where $\bar{\Omega
}=C([0,\infty);\mathbb{R}^{d})$ equipped with the usual local uniform
convergence topology, such that, the canonical process $\bar{B}_{t}%
(\bar{\omega})=\bar{\omega}_{t}$, $t\geq0$ is a Gaussin process such that
\[
G_{\underline{t}}^{\bar{B}}=G_{\underline{t}}^{B},\  \  \underline{t}%
\in \mathcal{T}.
\]
Readers who are interested in the details can see our appendix.
\end{example}

\begin{definition}
Let $(X_{t})_{t\geq0}$ and $(Y_{t})_{t\geq0}$ be two stochastic processes in a
nonlinear expectation space $(\Omega,\mathcal{H},\mathbb{E})$. They are called
identically distributed if $\mathbb{F}_{\underline{t}}^{X}=\mathbb{F}%
_{\underline{t}}^{Y}$, for each $\underline{t}\in \mathcal{T}$. $(Y_{t}%
)_{t\geq0}$ is said to be distributionally independent of another process
$(Z_{t})_{t\geq0}$ if, for each $\underline{t}=(t_{1},\cdots,t_{n}%
)\in \mathcal{T}$ , $(Y_{t_{1}},\cdots,Y_{t_{n}})$ is independent of
$(Z_{t_{1}},\cdots,Z_{t_{n}})$.
\end{definition}

\begin{definition}
A sequence of $d$-dimensional stochastic processes $\{(X_{t}^{i})_{t\geq
0}\}_{i=1}^{\infty}$ in a nonlinear expectation space $(\Omega,\mathcal{H}%
,\mathbb{E})$ is said to be convergence in fintinte dimensional distributions
if for each $\underline{t}=(t_{1},\cdots,t_{n})\in \mathcal{T}$ and for each
$\varphi \in C_{b.Lip}(\mathbb{R}^{n\times d})$, the limit $\lim_{i\rightarrow
\infty}\{ \mathbb{E}[\varphi(X_{t_{1}}^{i},\cdots,X_{t_{n}}^{i})]$ exists.
\end{definition}

\begin{theorem}
Let $\{(X_{t}^{i})_{t\geq0}\}_{i=1}^{\infty}$ be a sequence of $d$-dimensional
stochastic processes in a sublinear expectation space $(\Omega,\mathcal{H}%
,\mathbb{E})$ such that, for each $i=1,2,\cdots,$\newline(i) $(X_{t}%
^{i})_{t\geq0}$ and $(X_{t}^{1})_{t\geq0}$ are identically
distributed;\newline(ii) $(X_{t}^{i+1})_{t\geq0}$ is distributonally
independent of $(X_{t}^{1},\cdots,X_{t}^{i+1})_{t\geq0}$;\newline(iii)
$\mathbb{E}[X_{t}^{i}]=\mathbb{E}[-X_{t}^{i}]\equiv0$, $t\geq0$. \newline Then
the sum%
\[
Z_{t}^{(N)}:=\frac{1}{\sqrt{N}}\sum_{i=1}^{Nn}X_{t}^{i},\  \ t\geq0,
\]
converges in fintinte dimensional distributions to a Gaussian process
$(Z_{t})_{t\geq0}$ under a sublinear expectation space $(\bar{\Omega
},\mathcal{\bar{H}},\mathbb{\bar{E}})$. The family of the 2nd moments of
$(X_{t}^{1})_{t\geq0}$ and that of $(Z_{t})_{t\geq0}$ are the same, namely
$\mathbb{E}[Z_{t}]=\mathbb{E}[-Z_{t}]\equiv0$ and $G_{\underline{t}}^{X^{1}%
}=G_{\underline{t}}^{Z}$, for each $\underline{t}\in \mathcal{T}$.
\end{theorem}

\begin{remark}
It is worth to stress here that, even in the case where the above
$\{(X_{t}^{i})_{t\geq0}\}_{i=1}^{\infty}$ is a sequence of $G$-Brownian
motions, the limit $(Z_{t})_{t\geq0}$ is a Gaussian process but it may not be
a $G$-Brownian motion.
\end{remark}

\begin{proof}
The proof is simply from the central limit theorem, i.e., Theorem \ref{tCLT}.
Indeed, for each fixed $\underline{t}=(t_{1},\cdots,t_{n})\in \mathcal{T}$,
$(X_{\underline{t}}^{i})_{i=1}^{\infty}:=(X_{t_{1}}^{i},\cdots,X_{t_{n}}%
^{i})_{i=1}^{\infty}$ is a sequence of $\mathbb{R}^{n\times d}$-valued random
vectors which is i.i.d. in the snese that $X_{\underline{t}}^{i}\overset{d}%
{=}X_{\underline{t}}^{1}$ and $X_{\underline{t}}^{i+1}$ is independent of
$X_{\underline{t}}^{1},\cdots,X_{\underline{t}}^{i}$, for $i=1,2,\cdots$. We
then can apply the central limit theorem under the sublinear expectation
$\mathbb{E}$ to prove that $Z_{\underline{t}}^{(N)}:=\frac{1}{\sqrt{N}}%
\sum_{i=1}^{N}X_{\underline{t}}^{i}$ converges in law to an $\mathbb{R}%
^{n\times d}$-valued random vector $Z_{\underline{t}}^{\ast}$, of which the
distribution denoted by $\mathbb{F}_{\underline{t}}^{\ast}$ is $G$-normal.
Moreover the family $(\mathbb{F}_{\underline{t}}^{\ast})_{\underline{t}%
\in \mathcal{T}}$ is consistent. It follows from Theorem \ref{Thm8.8} that
there exists a $d$-dimensional stochastic process $(Z_{t})_{t\geq0}$ in some
sublinear expectation space $(\bar{\Omega},\mathcal{\bar{H}},\mathbb{\bar{E}%
})$ such that $(\mathbb{F}_{\underline{t}}^{Z})_{\underline{t}\in \mathcal{T}%
}=(\mathbb{F}_{\underline{t}}^{\ast})_{\underline{t}\in \mathcal{T}}$. Thus
$(Z_{t})_{t\geq0}$ is a Gaussian process.
\end{proof}

\section{$q$-normal distribution and $q$-Brownain motion in quantum mechanics}

The approach to construct stochastic processes such as $G$-Brownian motion,
$G$-Gaussian processes as well as some other typical stochastic processes in a
nonlinear expectation, e.g. L\'{e}vy processes and Markovian processes, can be
also applied to construct some new stochastic processes in an $\mathbb{R}^{n}%
$-valued expectation space $(\Omega,\mathcal{H},\mathbb{E})$. As a very
typical example we explain how to construct a $q$-Brownain motion under a
$\mathbb{C}$-valued linear expectation space.

Let $\Omega$ be a given set and let $\mathcal{H}$ be a linear space of complex
valued functions defined on $\Omega$ such that $c\in \mathcal{H}$ for each
complex constant $c$. The space $\mathcal{H}$ is the random space in our consideration.

\begin{definition}
{{A $\mathbb{C}$-valued \textbf{linear expectation }$\mathbb{E}$ is a
functional $\mathbb{{E}}:\mathcal{H}\rightarrow \mathbb{C}$ satisfying }}

\noindent \textbf{\textup{(i)} Constant preserving:}
\[
\mathbb{E}[c]=c\  \  \  \text{for}\ c\in \mathbb{C}.
\]

\noindent \textbf{\textup{(ii)} Linearity: } For each $X,Y\in{\mathcal{H}}$,
\[
\mathbb{E}[\alpha X+\beta Y]=\alpha \mathbb{E}[X]+\beta \mathbb{E}%
[Y],\  \  \alpha,\beta \in \mathbb{C}%
\]

The triple $(\Omega,\mathcal{H},\mathbb{E}\mathbb{)}$ is called a $\mathbb{C}%
$-valued linear expectation space.
\end{definition}

Let $X_{1}$ and $X_{2}$ be two $\mathbb{C}^{m}$-valued random vectors defined
on a $\mathbb{C}$-valued linear{ expectation space }$(\Omega,\mathcal{H}%
,\mathbb{{E}})$. They are called identically distributed, denoted by
$X_{1}\overset{d}{=}X_{2}$, if, for each function $\varphi$ defined on
$\mathbb{C}^{m}$ such that $\varphi(X_{1})\in \mathcal{H}$ (resp.
$\varphi(X_{2})\in \mathcal{H}$) implies $\varphi(X_{2})\in \mathcal{H}$ (resp.
$\varphi(X_{1})\in \mathcal{H}$) and
\[
\mathbb{E}[\varphi(X_{1})]=\mathbb{E}[\varphi(X_{2})].
\]

\begin{definition}
In a $\mathbb{C}$-valued linear expectation space $(\Omega,\mathcal{H}%
,\mathbb{E})$, for two random vectors $X\in \mathcal{H}^{m}$ and $Y\in
\mathcal{H}^{n}$, $Y$ is said to be independent of $X$ under $\mathbb{E}$ if
we have
\[
\mathbb{E}[\varphi(X,Y)]=\mathbb{E}[\mathbb{E}[\varphi(x,Y)]_{x=X}],
\]
for each function $\varphi$ on $\mathbb{C}^{m+n}$ such that the above
operations of expectations are meaningful. $Y$ is said to be an independent
copy of $X$ if moreover $Y\overset{d}{=}X$.
\end{definition}

We refer to [Peng2010-chI] for the product space method to construct
independent random variables with specific distributions.

\begin{definition}
An $\mathbb{R}^{d}$-valued valued random vector $X=(X_{1},\cdots,X_{d})$ on a
linear valuation space $(\Omega,\mathcal{H},\mathbb{E})$ is called a standard
$q$\textbf{-}normal distributed if its components are independent from each
others with $\varphi(X)\in \mathcal{H}$ and $\mathbb{E}[X_{k}^{2}]=-i$ ($i$
stands for the imaginary number) and
\begin{equation}
aX+b\bar{X}\overset{d}{=}\sqrt{a^{2}+b^{2}}X\  \  \  \text{for }a,b\in
\mathbb{R},\  \label{q-normal}%
\end{equation}
where $\bar{X}$ is an independent copy of $X$. Here $h(\mathbb{R}^{d})$ is the
space of complex valued functions on $\mathbb{R}^{d}$ spanned by polynomials
of $(x_{1},\cdots,x_{d})$ and all $\varphi \in C^{\infty}(\mathbb{R}^{d})$ such
that $\partial_{x}^{(n)}\varphi \in L^{2}(\mathbb{R}^{d})$, $n=0,1,2,\cdots.$
\end{definition}

\begin{remark}
Noting that $\mathbb{E}[X_{k}+\bar{X}_{k}]=2\mathbb{E}[X_{k}]$ and
$\mathbb{E}[X_{k}+\bar{X}_{k}]=\mathbb{E}[\sqrt{2}X_{k}]=\sqrt{2}%
\mathbb{E}[X_{k}]$, we then have $\mathbb{E}[X_{k}]=0$, $k=1,\cdots,d.$
\end{remark}

Just like in the case of normal distribution we can also define, for a fixed
complex valued function $\varphi \in h(\mathbb{C}^{d})$,
\[
w(t,x):=\mathbb{E}[\varphi(x+\sqrt{t}X)].
\]
From the definition (\ref{q-normal}) we have
\begin{align*}
w(t,x)  &  =\mathbb{E}[\varphi(x+\sqrt{\delta}X+\sqrt{t}\bar{X})]\\
&  =\mathbb{E}[w(t,x+\sqrt{\delta}X)].
\end{align*}
Similar to the situation of $G$-normal distributions, the function $w$ solves
the following free Schr\"{o}dinger equation (for $d=1$):
\[
\partial_{t}w(t,x)=\frac{i}{2}\Delta w(t,x).
\]
We can use the same method as in [Peng2010] to prove the existance of such
type of $q$-normal distributed random variable $X$, using classical results of PDE.

\begin{definition}
We call $(X_{t}(\omega))_{t\geq0}$ a complex-valued stochastic process define
in a linear expectation space $(\Omega,\mathcal{H},\mathbb{E})$ if $X_{t}%
\in \mathcal{H}$ for each $t\geq0$.
\end{definition}

\begin{definition}
A stochastic process $(B_{t})_{t\geq0}$ define in a $\mathbb{C}$-valued linear
valuation space $(\Omega,\mathcal{H},\mathbb{E})$ is call a $q$-Brownian
motion if it satisfies: for each $t,s\geq0$, \newline(i) $B_{t+s}%
-B_{s}\overset{d}{=}B_{t}$ and $B_{t+s}-B_{s}$ is independent of $B_{t_{1}%
},\cdots,B_{t_{n}}$, $t_{i}\leq s$, $i=1,2,\cdots$;\newline(ii) $\mathbb{E}%
[B_{t}]\equiv0$ and $\mathbb{E}[B_{t}^{2}]=-it$.
\end{definition}

In analogous to $G$-Brownian motions we can the construction a $q$-Brownian
motion as follows. Let $\Omega=(\mathbb{R}^{d})^{[0,\infty)}$ be the space of
all $d$-dimensional complexed valued process and let $B_{t}(\omega)=\omega
_{t}$, $t\geq0$, be the canonical process. We define
\[
\mathcal{H}=\{ \varphi(B_{t_{1}},\cdots,B_{t_{n}}):\underline{t}=(t_{1}%
,\cdots,t_{n})\in \mathcal{T},\  \varphi \in h(\mathbb{R}^{n\times d})\}.
\]
It then remains to construct consistently a $\mathbb{C}$-valued linear
expectation of $\mathbb{E}$ on $(\Omega,\mathcal{H})$ under which the
canonical process $(B_{t})_{t\geq0}$ is a $q$-Brownian motion. To this end we
are given a sequence of standard $q$-normally distributed random variables
$\{X_{i}\}_{i=1}^{\infty}$ of a $\mathbb{C}$-valued expectation space
$(\bar{\Omega},\mathcal{\bar{H}},\mathbb{\bar{E}})$ such that $X_{i+1}$ is
independent of $(X_{1},\cdots,X_{i})$ for $i=1,2,\cdots$. For each
$\underline{t}=(t_{1},\cdots,t_{n})\in \mathcal{T}$ with $t_{1}\leq t_{2}%
\leq \cdots \leq t_{n}$, we set, for each $\xi \in \mathcal{H}$ of the form
$\xi(\omega)=\varphi(B_{t_{1}},B_{t_{2}}-B_{t_{1}},\cdots,B_{t_{n}}-B_{t_{n}%
})$, we set
\begin{align*}
\mathbb{E}[\xi]  &  =\mathbb{E}[\varphi(B_{t_{1}},B_{t_{2}}-B_{t_{1}}%
,\cdots,B_{t_{n}}-B_{t_{n}})]\\
&  :=\mathbb{\bar{E}}[\varphi(\sqrt{t_{1}}X_{1},\sqrt{t_{2}-t_{1}}X_{2}%
,\cdots,\sqrt{t_{n}-t_{n-1}}X_{n})].
\end{align*}
We see that $\mathbb{E}:\mathcal{H}\mapsto \mathbb{C}$ consistently defines a
$\mathbb{C}$-valued linear expectation under which $(B_{t})_{t\geq0}$ becomes
a $q$-Brownian motion.

We can also check that $w(t,x):=\mathbb{E}[\varphi(x+B_{t})]$, $t\geq0$,
$x\in \mathbb{R}^{d}$ solves the following free Schr\"{o}dinger equation%
\[
\partial_{t}w(t,x)=\frac{i}{2}\Delta w(t,x),\  \ w|_{t=0}=\varphi.
\]
The situation with potential $V(x)$ corresponds to:
\[
w(t,x)=\mathbb{E}[\varphi(x+B_{t})\exp \{ \int_{0}^{t}V(x+B_{s}%
)ds\}]\},\  \ t\geq0,x\in \mathbb{R}^{d}.
\]
This forms a new type of Feynman-Kac formula to give a path-representation of
the solution of a Schr\"{o}dinger equation.

\section{Appedix}

\subsection{Appendix A: Parabolic PDE associated with $G$-normal
distributions}

The distribution of a $G$-normally distributed random vector $X$ is
characterized by the following parabolic partial differential equation (PDE
for short) defined on $[0,\infty)\times \mathbb{R}^{d}:$%
\begin{equation}
\partial_{t}u-G(D_{x}^{2}u)=0, \label{ee03}%
\end{equation}
with Cauchy condition$\  \ u|_{t=0}=\varphi$, where $G:\mathbb{S}%
(d)\rightarrow \mathbb{R}$ is defined by (\ref{e313}) and $D^{2}u=(\partial
_{x_{i}x_{j}}^{2}u)_{i,j=1}^{d}$, $Du=(\partial_{x_{i}}u)_{i=1}^{d}$. The PDE
(\ref{ee03}) is called a \textbf{$G$-heat equation}%
\index{$G$-equation}%
.

\begin{proposition}
\label{gG-P1 copy(1)}Let $X\in \mathcal{H}^{d}$ be normally distributed, i.e.,
\textup{(\ref{e311})} holds. Given a function $\varphi \in C_{l.Lip}%
(\mathbb{R}^{d})$, we define
\[
u(t,x):=\mathbb{E}[\varphi(x+\sqrt{t}X)],\ (t,x)\in \lbrack0,\infty
)\times \mathbb{R}^{d}.
\]
Then we have%
\begin{equation}
u(t+s,x)=\mathbb{E}[u(t,x+\sqrt{s}X)],\  \ s\geq0. \label{e315}%
\end{equation}
We also have the estimates: for each $T>0,$ there exist constants $C,k>0$ such
that, for all $t,s\in \lbrack0,T]$ and $x,\bar{x},y,\bar{y}\in \mathbb{R}^{d}$,
\begin{equation}
|u(t,x)-u(t,\bar{x})|\leq C(1+|x|^{k}+|\bar{x}|^{k})(|x-\bar{x}|) \label{e04}%
\end{equation}
and%
\begin{equation}
|u(t,x)-u(t+s,x)|\leq C(1+|x|^{k})(s+|s|^{1/2}). \label{e05}%
\end{equation}
Moreover, $u$ is the unique viscosity solution, continuous in the sense of
\textup{(\ref{e04})} and \textup{(\ref{e05})}, of the PDE \textup{(\ref{ee03}%
)}.
\end{proposition}

\begin{corollary}
\label{gG-P1coro copy(1)} If both $X$ and $\bar{X}$ satisfy
\textup{(\ref{e311})} {with the same }$G${, i.e.,}%
\[
G(A):=\mathbb{E}[\frac{1}{2}\left \langle AX,X\right \rangle ]=\mathbb{E}%
[\frac{1}{2}\left \langle A\bar{X},\bar{X}\right \rangle ]\  \  \  \text{for}%
\ A\in \mathbb{S}(d),
\]
{ then }$X\overset{d}{=}\bar{X}$. In particular, $X\overset{d}{=}-X$.
\end{corollary}

\begin{example}
Let $X$ be $G$-normal distributed. The distribution of $X$ is characterized
by
\[
u(t,x)=\mathbb{E}[\varphi(x+\sqrt{t}X)],\  \  \varphi \in C_{l.Lip}%
(\mathbb{R}^{d}).
\]
In particular, $\mathbb{E}[\varphi(X)]=u(1,0)$, where $u$ is the solution of
the following parabolic PDE defined on $[0,\infty)\times \mathbb{R}^{d}:$%
\begin{equation}
\partial_{t}u-G(D^{2}u)=0,\  \ u|_{t=0}=\varphi, \label{e03}%
\end{equation}
where $G=G_{X}(A):\mathbb{S}(d)\rightarrow \mathbb{R}$ is defined by
\[
G(A):=\frac{1}{2}\mathbb{E}[\left \langle AX,X\right \rangle ],\  \  \ A\in
\mathbb{S}(d).
\]

The parabolic PDE \textup{(\ref{e03})} is called a \textbf{$G$-heat equation}.%
\index{$G$-heat equation}%

It is easy to check that $G$ is a sublinear function defined on $\mathbb{S}%
(d)$. By Theorem \ref{t1} there exists a bounded, convex and closed subset
$\Theta \subset \mathbb{S}(d)$ such that
\[
\frac{1}{2}\mathbb{E}[\left \langle AX,X\right \rangle ]=G(A)=\frac{1}{2}%
\sup_{Q\in \Theta}tr[AQ],\  \ A\in \mathbb{S}(d).
\]
Since $G(A)$ is monotonic: $G(A_{1})\geq G(A_{2})$, for $A_{1}\geq A_{2}$, it
follows that%
\[
\Theta \subset \mathbb{S}_{+}(d)=\{ \theta \in \mathbb{S}(d):\theta \geq
0\}=\{BB^{T}:B\in \mathbb{R}^{d\times d}\},
\]
where $\mathbb{R}^{d\times d}$\label{splusd} is the set of all $d\times d$
matrices. If $\Theta$ is a singleton: $\Theta=\{Q\}$, then $X$ is classical
zero-mean normal distributed with covariance $Q$. In general, $\Theta$
characterizes the covariance uncertainty of $X$.

When $d=1$, we have $X\overset{d}{=}N(\{0\} \times \lbrack \underline{\sigma
}^{2},\bar{\sigma}^{2}])$ (We also denoted by $X\overset{d}{=}N(0,[\underline
{\sigma}^{2},\bar{\sigma}^{2}])$), where $\bar{\sigma}^{2}=\mathbb{E}[X^{2}]$
and $\underline{\sigma}^{2}=-\mathbb{E}[-X^{2}]$. The corresponding $G$-heat
equation is
\[
\partial_{t}u-\frac{1}{2}(\bar{\sigma}^{2}(\partial_{xx}^{2}u)^{+}%
-\underline{\sigma}^{2}(\partial_{xx}^{2}u)^{-})=0,~u|_{t=0}=\varphi.
\]
For the case $\underline{\sigma}^{2}>0$, this equation is also called the
Barenblatt equation.
\end{example}

In the following two typical situations, the calculation of $\mathbb{E}
[\varphi(X)]$ is very easy:

\begin{itemize}
\item For each \textbf{convex} function $\varphi$, we have
\[
\mathbb{E}[\varphi(X)]=\frac{1}{\sqrt{2\pi}}\int_{-\infty}^{\infty}%
\varphi(\overline{\sigma}^{2}y)\exp(-\frac{y^{2}}{2 })dy.
\]
Indeed, for each fixed $t\geq0$, it is easy to check that the function
{$u(t,x):=\mathbb{E}[\varphi(x+\sqrt{t}X)]$ is convex in $x$:}
\begin{align*}
u(t,\alpha x+(1-\alpha)y)  &  =\mathbb{E}[{\varphi(\alpha x+(1-\alpha
)y+\sqrt{t}X)]}\\
&  \leq{}{{\alpha}\mathbb{E}[\varphi(x+\sqrt{t}X)]+(1-{\alpha)}\mathbb{E}%
[\varphi(x+\sqrt{t}X)]}\\
&  =\alpha u(t,x)+(1-\alpha)u(t,x).
\end{align*}
It follows that $(\partial_{xx}^{2}u)^{-}\equiv0$ and thus the above $G$-heat
equation becomes {%
\[
\partial_{t}u=\frac{\overline{\sigma}^{2}}{2}\partial_{xx}^{2}u,\  \  \ u|_{t=0}%
=\varphi.\  \
\]
}

\item For each \textbf{concave} function $\varphi$, we have%
\[
\mathbb{E}[\varphi(X)]=\frac{1}{\sqrt{2\pi}}\int_{-\infty}^{\infty}%
\varphi(\underline{\sigma}^{2}y)\exp(-\frac{y^{2}}{2})dy.
\]

In particular,%
\[
\mathbb{E}[X]=\mathbb{E}[-X]=0,\  \  \mathbb{E}[X^{2}]=\overline{\sigma}%
^{2},\  \ -\mathbb{E}[-X^{2}]=\underline{\sigma}^{2}.
\]

\end{itemize}

\subsection{Appendix B: Kolmogorov's criterion in the situation of sublinear
expectation}

\begin{definition}
Let $I$ be a set of indices, $(X_{t})_{t\in I}$ and $(Y_{t})_{t\in I}$ be two
processes indexed by $I$ . We say that $Y$ is a \textbf{ modification}%
\index{Modification}
of $X$ if for all $t\in I$, $X_{t}=Y_{t}$ q.s..
\end{definition}

We now give a Kolmogorov criterion for a process indexed by $\mathbb{R}^{d}$
with $d\in \mathbb{N}$.

\begin{theorem}
\label{ch6t128} Let $p>0$ and $(X_{t})_{t\in \lbrack0,1]^{d}}$ be a process
such that for all $t\in \lbrack0,1]^{d}$, $X_{t}$ belongs to $\mathbb{L}^{p}$ .
Assume that there exist positive constants $c$ and $\varepsilon$ such that
\[
\bar{\mathbb{E}}[|X_{t}-X_{s}|^{p}]\leq c|t-s|^{d+\varepsilon},\  \ s,t\in
\lbrack0,1]^{d}%
\]
Then $X=(X_{t})_{t\in \lbrack0,1]^{d}}$ admits a modification $(\tilde
{X})_{t\in \lbrack0,1]^{d}}$ such that
\[
\bar{\mathbb{E}}\left[  \left(  \sup_{s\neq t}\frac{|\tilde{X}_{t}-\tilde
{X}_{s}|}{|t-s|^{\alpha}}\right)  ^{p}\right]  <\infty,
\]
for each $\alpha \in \lbrack0,\varepsilon/p)$. As a consequence, paths of
$\tilde{X}$ are quasi-surely H\"{o}lder continuous of order $\alpha$ for each
$\alpha<\varepsilon/p$ in the sense that there exists a Borel set $N$ of
$\tilde{c}(N)=0$ such that for all $\omega \in N^{c}$, the map $t\rightarrow
\tilde{X}_{t}(\omega)$ is H\"{o}lder continuous of order $\alpha$ for each
$\alpha<\varepsilon/p$. Moreover, if $X_{t}\in \mathbb{L}_{c}^{p}$ for each
$t$, then we also have $\tilde{X}_{t}\in \mathbb{L}_{c}^{p}$.
\end{theorem}

\begin{lemma}
\label{le6}Let $p>0$. Assume that there exist positive constants $c$ and
$\varepsilon$ such that
\[
\mathbb{F}_{t,s}[\varphi_{p}]\leq c|t-s|^{1+\varepsilon},\  \ s,t\geq0
\]
where $\varphi_{p}(x_{1},x_{2})=|x_{1}-x_{2}|^{p}$. Then for the stochastic
process $\overline{X}=(\overline{X}_{t})_{t\geq0}$, there exists a continuous
modification $\tilde{X}=\{ \tilde{X}_{t}:t\in \lbrack0,\infty)\}$ of
$\overline{X}$ (i.e. $\tilde{c}(\{ \tilde{X}_{t}\not =\overline{X}_{t}\})=0$,
for each $t\geq0$).
\end{lemma}

\begin{proof}
We have $\mathbb{\bar{E}}=\mathbb{\tilde{E}}$ on
$L_{ip}(\bar{\Omega})$. On the other hand, we have
\[
\mathbb{\tilde{E}}[|\overline{X}_{t}-\overline{X}_{s}|^{p}]=\mathbb{\bar{E}%
}[|\overline{X}_{t}-\overline{X}_{s}|^{p}]=c|t-s|^{1+\varepsilon},\forall
s,t\in \lbrack0,\infty),
\]
where $d$ is a constant depending only on $G$. By Theorem \ref{ch6t128}, there
exists a continuous modification $\tilde{B}$ of $\overline{X}$. Since
$\tilde{c}(\{ \overline{X}_{0}\not =0\})=0$, we can set $\tilde{B}_{0}=0$. The
proof is complete.
\end{proof}

\subsection{Appendix C: Tightness of $\mathcal{P}_{\! \!e}$}

For each $Q\in \mathcal{P}_{\! \!e}$, let $Q\circ \tilde{B}^{-1}$ denote the
probability measure on $(\Omega,\mathcal{B}(\Omega))$ induced by $\tilde{B}$
with respect to $Q$. We denote $\mathcal{P}_{1}=\{Q\circ \tilde{B}^{-1}%
:Q\in \mathcal{P}_{\! \!e}\}$. By Lemma \ref{le6}, we get
\[
\mathbb{\tilde{E}}[|\overline{X}_{t}-\overline{X}_{s}|^{p}%
]=c|t-s|^{1+\varepsilon},\forall s,t\in \lbrack0,\infty).
\]
Applying the well-known result of moment criterion for tightness of the above
Kolmogorov-Chentrov's type, we conclude that $\mathcal{P}_{1}$ is tight. We
denote by $\mathcal{P}=\overline{\mathcal{P}}_{1}$ the closure of
$\mathcal{P}_{1}$ under the topology of weak convergence, then $\mathcal{P}$
is weakly compact.

Now, we give the representation of $G$-expectation.

\begin{theorem}
\label{Gt34} For each continuous monotonic and sublinear function
$G:\mathbb{S}(d)\longmapsto \mathbb{R}$, let $\mathbb{\hat{E}}_{G}$ be the
corresponding $G$-expectation on $(\Omega,L_{ip}(\Omega))$. Then there exists
a weakly compact family of probability measures $\mathcal{P}$ on
$(\Omega,\mathcal{B}(\Omega))$ such that%

\[
\mathbb{\hat{E}}_{G}[X]=\max_{P\in \mathcal{P}}E_{P}[X],\quad \forall X\in
L_{ip}(\Omega).
\]

\end{theorem}

\begin{proof}
By Lemma \ref{le5} and Lemma \ref{le6}, we have
\[
\mathbb{\hat{E}}_{G}[X]=\max_{P\in \mathcal{P}_{1}}E_{P}[X],\quad \forall X\in
L_{ip}(\Omega).
\]
For each $X\in L_{ip}(\Omega)$, by Lemma \ref{le3}, we get $\mathbb{\hat{E}%
}_{G}[|X-(X\wedge N)\vee(-N)|]\downarrow0$ as $N\rightarrow \infty$. Noting
also that $\mathcal{P}=\overline{\mathcal{P}}_{1}$, then by the definition of
weak convergence, we get the result.
\end{proof}

\subsection{Appendix D: $G$-Capacity and paths of $G$-Brownian motion}

According to Theorem \ref{Gt34}, we obtain a weakly compact family of
probability measures $\mathcal{P}$ on $(\Omega,\mathcal{B}(\Omega))$ to
represent the sublinear expectation $\mathbb{E}[\cdot]$. For this
$\mathcal{P}$, we define the associated $G$-capacity:
\[
\hat{c}(A):=\sup_{P\in \mathcal{P}}P(A),\quad A\in \mathcal{B}(\Omega),
\]
and upper expectation for each $X\in L^{0}(\Omega)$ which makes the following
definition meaningful,
\[
\mathbb{\hat{E}}[X]:=\sup_{P\in \mathcal{P}}E_{P}[X].
\]
We have $\mathbb{\hat{E}}=\mathbb{\hat{E}}_{G}$ on $L_{ip}(\Omega)$.
Thus the $\mathbb{\hat{E}}_{G}[|\cdot|]$-completion and the
$\mathbb{\hat{E}}[|\cdot|]$-completion of $L_{ip}(\Omega)$ are the
same.

For each $T>0$, we also denote by $\Omega_{T}=C_{0}^{d}([0,T])$ equipped with
the distance
\[
\rho(\omega^{1},\omega^{2})=\left \Vert \omega^{1}-\omega^{2} \right \Vert
_{C_{0}^{d}([0,T])}:=\max_{0\leq t\leq T}|\omega^{1}_{t}-\omega^{2}_{t}|.
\]

We now prove that $L_{G}^{1}(\Omega)=\mathbb{L}_{c}^{1}$. First, we need the
following classical approximation Lemma (see e.g. the well-known approximation
of Barlow and Perkins for SDE and Lepeltier and San Martin to prove the
existence of BSDE with continuous coefficients).

\begin{proposition}
\label{pr11} For each $X\in C_{b}(\Omega)$ and $\varepsilon>0$, there exists a
$Y\in L_{ip}(\Omega)$ such that $\mathbb{\bar{E}}[|Y-X|]\leq \varepsilon$.
\end{proposition}

\begin{proof}
We denote $M=\sup_{\omega \in \Omega}|X(\omega)|$. We can find
$\mu>0$, $T>0$ and $\bar{X}\in C_{b}(\Omega _{T})$ such that
$\mathbb{\hat{E}}[|X-\bar{X}|]<\varepsilon/3$, $\sup _{\omega \in
\Omega}|\bar{X}(\omega)|\leq M$ and
\[
|\bar{X}(\omega)-\bar{X}(\omega^{\prime})|\leq \mu \left \Vert \omega
-\omega^{\prime}\right \Vert _{C_{0}^{d}([0,T])},\  \  \forall \omega
,\omega^{\prime}\in \Omega.
\]
Now for each positive integer $n$, we introduce a mapping $\omega^{(n)}%
(\omega):\Omega \mapsto \Omega$:
\[
\omega^{(n)}(\omega)(t)=\sum_{k=0}^{n-1}\frac{\mathbf{1}_{[t_{k}^{n}%
,t_{k+1}^{n})}(t)}{t_{k+1}^{n}-t_{k}^{n}}[(t_{k+1}^{n}-t)\omega(t_{k}%
^{n})+(t-t_{k}^{n})\omega(t_{k+1}^{n})]+\mathbf{1}_{[T,\infty)}(t)\omega(t),\
\]
where $t_{k}^{n}=\frac{kT}{n},\ k=0,1,\cdots,n$. We set $\bar{X}^{(n)}%
(\omega):=\bar{X}(\omega^{(n)}(\omega))$, then
\begin{align*}
|\bar{X}^{(n)}(\omega)-\bar{X}^{(n)}(\omega^{\prime})|  &  \leq \mu \sup
_{t\in \lbrack0,T]}|\omega^{(n)}(\omega)(t)-\omega^{(n)}(\omega^{\prime})(t)|\\
&  =\mu \sup_{k\in \lbrack0,\cdots,n]}|\omega(t_{k}^{n})-\omega^{\prime}%
(t_{k}^{n})|.
\end{align*}
We now choose a compact subset $K\subset \Omega$ such that $\mathbb{\hat{E}%
}[\mathbf{1}_{K^{C}}]\leq \varepsilon/6M$. Since $\sup_{\omega \in K}\sup
_{t\in \lbrack0,T]}|\omega(t)-\omega^{(n)}(\omega)(t)|\rightarrow0$, as
$n\rightarrow \infty$, we then can choose a sufficiently large $n_{0}$ such
that
\begin{align*}
\sup_{\omega \in K}|\bar{X}(\omega)-\bar{X}^{(n_{0})}(\omega)|  &
=\sup_{\omega \in K}|\bar{X}(\omega)-\bar{X}(\omega^{(n_{0})}(\omega))|\\
&  \leq \mu \sup_{\omega \in K}\sup_{t\in \lbrack0,T]}|\omega(t)-\omega^{(n_{0}%
)}(\omega)(t)|\\
&  <\varepsilon/3.
\end{align*}
Set $Y:=\bar{X}^{(n_{0})}$, it follows that
\begin{align*}
\mathbb{\bar{E}}[|X-Y|]  &  \leq \mathbb{\bar{E}}[|X-\bar{X}|]+\mathbb{\bar{E}%
}[|\bar{X}-\bar{X}^{(n_{0})}|]\\
&  \leq \mathbb{\bar{E}}[|X-\bar{X}|]+\mathbb{\bar{E}}[\mathbf{1}_{K}|\bar
{X}-\bar{X}^{(n_{0})}|]+2M\mathbb{\bar{E}}[\mathbf{1}_{K^{C}}]\\
&  <\varepsilon.
\end{align*}
The proof is complete.
\end{proof}

\  \  \

By Proposition \ref{pr11}, we can easily get $L_{G}^{1}(\Omega)=\mathbb{L}%
_{c}^{1}$. Furthermore, we can get $L_{G}^{p}(\Omega)=\mathbb{L}_{c}^{p}$,
$\forall p>0$.

Thus, we obtain a pathwise description of $L_{G}^{p}(\Omega)$ for each $p>0$:
\[
L_{G}^{p}(\Omega)=\{X\in L^{0}(\Omega):X\  \text{is quasi-continuous and }%
\lim_{n\rightarrow \infty}\hat{\mathbb{E}}[|X|^{p}I_{\{|X|>n\}}]=0\}.
\]
Furthermore, $\mathbb{\hat{E}}_{G}[X]=\hat{\mathbb{E}}[X]$, for each $X\in
L_{G}^{1}(\Omega)$.\newline

\end{document}